\documentclass[12pt]{amsart}
\usepackage{amsmath}
\usepackage{amsthm}
\usepackage{amsfonts}
\usepackage{amssymb}
\usepackage{graphicx}

\theoremstyle{plain}
\newtheorem{thm}{Theorem}[section]

\newtheorem{lem}[thm]{Lemma}
\newtheorem{cor}[thm]{Corollary}

\theoremstyle{definition}
\newtheorem{defn}{Definition}

\theoremstyle{remark}
\newtheorem{remark}{Remark}

\advance \topmargin by -\headheight
\advance \topmargin by -\headsep
\evensidemargin \oddsidemargin
\title{On The Number of Edge-3-Colourings of A Snipped Snark}
\author{Scott A. McKinney}
\date{\today}
\usepackage{array}  
\usepackage[all]{xy}

\begin{document}

\begin{abstract}
For a given snark $G$ and edge $e$ of $G$, we can form a cubic graph $G_e$ using an operation we call "edge subtraction".  The number of 3-edge-colourings of $G_e$ is $18 \cdot \psi(G,e)$ for some nonnegative integer $\psi(G,e)$. Given snarks $G_1$ and $G_2$, we can form a new snark $G$ using techniques given by Isaacs \cite{Is} and Kochol \cite{Ko}.  In this note we give relationships between $\psi(G_1,e_1)$, $\psi(G_2,e_2)$, and $\psi(G,e)$ for particular edges $e_1$, $e_2$, and $e$, in $G_1$, $G_2$, and $G$ (respectively). As a consequence, if $g,h,i,j,k,l$ are each a nonnegative integer, then there exists a cyclically 5-edge-connected snark $G$ with an edge $e$ such that $\psi(G,e)=5^g \cdot 7^h$, and a cyclically 4-edge-connected snark $G_0$ with an edge $e_0$ such that $\psi(G_0,e_0)=2^i \cdot 3^j \cdot 5^k \cdot 7^l$.
\end{abstract}

\maketitle


\section{Introduction}

This paper pertains only to undirected graphs. A graph is "cubic" if each vertex is attached to 3 edges, and is "simple" if it is connected, has finitely many vertices, and has no loops or multiple edges.  A graph is "quasi-cubic" if each vertex is attached to 1 or 3 edges, and is "non-trivial" if it has at least one 3-valent vertex. A 3-edge-decomposition of a simple, non-trivial quasi-cubic graph is a partitioning of the edges into three classes in such a way that no two adjacent edges are in the same class. A 3-edge-colouring of a simple, non-trivial quasi-cubic graph is an assignment of three colours to each edge of the graph in such a way that no two adjacent edges are assigned the same colour.  Note that for every 3-edge-decomposition of a graph $G$, there are $3!=6$ 3-edge-colourings of $G$.\medskip
An important family of cubic graphs is the snark family.\medskip
\begin{defn}[Snark]
A snark $G$ is a simple, non-3-edge-colourable, cubic graph, with girth of at least 5, and at least cyclically 4-edge-connected.\medskip
\end{defn}

See Gardner \cite{Ga} for background information on how snarks arose during investigations surrounding the 4-color map theorem.  Also see Bradley \cite{Br2} and Hagglund \cite{Ha} for further motivations and open questions.\medskip
\begin{defn}[Edge elimination]Given an edge $e$ in a snark $G$, one can eliminate $e$ from $G$ to obtain a non-cubic graph $G-\{e\}$ by removing the edge $e$, and leaving the two divalent vertices to which $e$ was attached.
\end{defn}
\begin{defn}[Edge subtraction]
\label{defn:Edge subtraction} Given an edge $e=(u,v)$ in a cubic graph
$G$, one can "subtract" the edge $e$ from $G$ to obtain a cubic graph $G_e$.
First, remove the edge $e$.  This leaves divalent vertices $u$ and $v$, with edges $(u,u_1), (u,u_2), (v,v_1), (v,v_2)$ . Delete vertices $u$ and $v$ and the attached edges, and form new edges $(u_1, u_2)$ and $(v_1, v_2)$, to obtain a new cubic graph $G_e$.\medskip
\end{defn}

\begin{defn}[Kempe chain]
\label{defn:Kempe chain} A Kempe chain on a simple, colourable, non-trivial quasi-cubic graph $G$ is a cycle of edges of $G$ with edges given two distinct colours, necessarily in an alternating way.\medskip
\end{defn}

This note is concerned with the number of colourings of the "snark minus an edge" $G_e$, where $G$ is a snark and $e$ is a specified edge of $G$. We shall first establish a number of facts about $G_e$ where $G$ is any simple, non-colourable cubic graph.

\section {Some Useful Lemmas}

The following facts will be useful for counting the number of colourings of a snark minus an edge, as defined above.  Unless otherwise noted, all theorems and lemmas appear in the following context.\medskip
Context 1: Let $G$ be any simple, non-colourable cubic graph, and let $e$ be any edge of $G$. Let $C$ be a set of 3 colours $\{a,b,c\}$, and $C(G_e)$ be the (possibly empty) set of colourings of $G_e$ with colours from $C$.  Let $\# \{\dots\}$ denote the cardinality of a set $\{\dots\}$.  Let $d$ be an edge of $G_e$ and let $f$ be a colouring of $G_e$ such that $f(d)$ = $x$ or $y$, where $x$ and $y$ are distinct choices of colours of $C$. Let $K(d,x,y)$ denote the Kempe Chain of colours $x$, $y$, containing edge $d$. Let $d_1$, $d_2$ denote the two new edges of $G_e$ resulting from the edge subtraction of $e$ from $G$.\medskip
\begin{lem}Let $e_1, e_2$ be edges of $G$, and suppose there is no 3-edge-decomposition of $G_e$ in which $e_1$ and $e_2$ are in the same class.  Then the number of 3-edge-decompositions of $G_e$ is precisely the number of colourings $F$ of $G_e$ such that $F(e_1)=x$ and $F(e_2)=y$, where $x$ and $y$ are distinct choices of colours of C.\medskip
\end{lem}

\begin{proof}Each such colouring $F$ corresponds to an edge-decomposition where the edges in the class containing $e_1$ are coloured $x$ and the edges in the class containing $e_2$ are coloured $y$.\medskip
\end{proof}

\begin{remark}If $e_1, e_2$ are adjacent, then Lemma 2.1 applies.\medskip
\end{remark}

\begin{thm}\#\{3-edge-decompositions($G_e$)\}=3J for some nonnegative integer J (possibly 0).\medskip
\end{thm}

Theorem 2.2 was proved by K\'{a}szonyi (\cite{Ka}, \cite{Kz}) as follows.  Lemma 2.3 was proved by K\'{a}szonyi (\cite{Ka}, pages 124-125).  Lemma 2.4 was proved by K\'{a}szonyi (\cite{Kz}, page 35, paragraph after Theorem 4.3). The proofs given below for those lemmas are K\'{a}szonyi's arguments in a convenient form.\medskip

\begin{lem}If $G_e$ is given a colouring $f$ such that f($d_1$) = x or y, then $d_2$ $\notin$ K($d_1$,x,y).\medskip
\end{lem}

\begin{proof}Suppose to the contrary that $d_2$ $\in$ K($d_1$, x, y).  Reinsert the edge $e$; this separates K($d_1$, x, y) into two paths $K_1$ and $K_2$.  On $K_1$, interchange the colours $x$ and $y$, and assign to $e$ the third colour $z$.  $G$ is now coloured and we have a contradiction.\medskip
\end{proof}

\begin{defn}
\label{defn:C'} Let $v$ be an endpoint of $d_2$, with $e_1$ and $e_2$ the other two edges adjoining $v$.  Let $C'(G_e)$=$\{f \in C(G_e):f(d_2,e_1,e_2)=abc\}$. Let $C_x'$=$\{f \in C'(G_e):f(d_1)=x\}$.\medskip
\end{defn}

\begin{lem}For some nonnegative integer J, \#$C_x'(G_e)=J$ $\forall$$x$$\in$$C$.\medskip
\end{lem}

\begin{proof}By Lemma 2.2, if $f \in C_a'$, then $K(d_1,a,b)$ does not contain $d_2$.  Hence, we can interchange the colours $a, b$ along
$K(d_1,a,b)$ to obtain a colouring $f_b$ of $G_e$ with $f_b(d_1)=b$.
This gives us a 1-1 correspondence between $C_a'$ and $C_b'$.
Similarly, there is a 1-1 correspondence between $C_a'$ and $C_c'$.
Hence, $\#C_a'=\#C_b'=\#C_c'$.\medskip
\end{proof}

\begin{proof}[Proof of Theorem 2.2]
From Lemma 2.4, \#$C'(G_e)=3J$.  Since we have fixed the colours of
$d_2,e_1,$ and $e_2$, each colouring corresponds to a
3-edge-decomposition of $G_e$ by Lemma 2.1. Theorem 2.2 follows.\medskip
\end{proof}

\begin{cor}\# \{3-edge-colourings($G_e$)\}=18J.\medskip
\end{cor}

\begin{proof}For each 3-edge-decomposition of $G_e$, there are
exactly $3!=6$ 3-edge-colourings of $G_e$.\medskip
\end{proof}

\begin{defn}[The nonnegative integers $\psi$(G,e)]
\label{defn:Psi} Since J depends on the snark $G$ and the edge $e$,
we can define the number $\psi(G,e)$:=J.\medskip
\end{defn}

\begin{thm}[Parity Lemma]Let $H$ be any simple, colourable quasi-cubic
graph.  Consider any colouring of $H$.  The number of vertices of
$H$ that are not connected to an edge of given colour is even.
Equivalently, the number of univalent vertices of $H$ not connected
to an edge of given colour is even.\medskip
\end{thm}

\begin{lem}The number of vertices of $H$ connected to an edge of
given colour $x$ is even.\medskip
\end{lem}

\begin{proof}Exactly 2 vertices of $H$ are attached to each edge of
colour $x$.  Furthermore, no vertex can be attached to more than one
edge of colour $x$, so the number of vertices attached to an edge of
colour $x$ is twice the number of edges of colour $x$.\medskip
\end{proof}

\begin{lem}The number of vertex-edge connections of $H$ is
even.\medskip
\end{lem}

\begin{proof}Each edge corresponds to exactly 2 unique vertex-edge connections.\medskip
\end{proof}

\begin{lem}The number of vertices of $H$ is even.\medskip
\end{lem}

\begin{proof}Let $v_1,\dots,v_n$ be the vertices of H.  Let
$E(v_i)$ (either 1 or 3) be the number of edges connected to $v_i$. Then $\sum_{i=1}^n E(v_i)\equiv 0 \pmod 2$ by Lemma 2.8. But the sum of odd terms is even if and only if the number of terms being summed is even.  Hence n is even.\medskip
\end{proof}

\begin{proof}[Proof of Theorem 2.6]
The number of vertices of $H$ not connected to an edge of colour $x$ is precisely the number of vertices of $H$ (even by Lemma 2.9) minus the number of vertices of $H$ connected to an edge of colour $x$ (even by Lemma 2.7).  Theorem 2.6, the "Parity Lemma", follows.\medskip
\end{proof}

There is a well-known reformulation of the Parity Lemma in terms of the group $\mathbb{Z}_2$x$\mathbb{Z}_2$ = $\{(0,0),(0,1),(1,0),(1,1)\}$=$\{0,a,b,c\}$, with the group operation $+$ defined coordinate-wise in $\mathbb{Z}$ so that $a+a=0$, etc., $a+b=c, a+c=b, b+c=a$, $a+0=a$ , etc., and $a+b+c=0$.\medskip

\begin{defn}
Given a colouring $f$ of $H$, for each vertex $v_i$ of $H$, let $\varsigma(v_i)$ be the sum of the colours of edges that meet $v_i$.  Notice that for trivalent vertices of $H$, $\varsigma(v_i)=0$, since all three colours $a,b,c$ are included in the sum.\medskip
\end{defn}

\begin{thm}[Parity Lemma, A Reformulation]If $f$ is a colouring of $H$,
then $\sum_{i=1}^n \varsigma(v_i)=0$.  Since $\varsigma(v_i)=0$ for all trivalent
vertices $v_i$, then the sum in this reformulation needs to be taken
only over the univalent vertices of $H$.\medskip
\end{thm}

\begin{defn}[Cut Set, Minimal Cut Set]A cut set of a snark $G$ is a set of edges
$e_1,\dots,e_n$ for which the graph $G-{e_1}-\dots-{e_n}$ contains
two disjoint graphs.  A minimal (in the sense of inclusion) cut set
is a cut set $e_1,\dots,e_n$ with no proper non-empty subset
$e_i,\dots,e_j$ such that $e_i,\dots,e_j$ is a cut set.\medskip
\end{defn}

\begin{defn}[Parity Lemma on a Minimal Cut Set]The parity lemma is especially useful on a minimal cut set of $G_e$.  Consider a minimal cut set $e_1,\dots,e_n$ of a colourable snark minus an edge, $G_e$. Let $G_1$ and $G_2$ denote the two subgraphs of $G_e$ obtained by eliminating $e_1,\dots,e_n$ (recall the definition of edge elimination). If, instead of eliminating the edges $e_1,\dots,e_n$, we eliminate $G_1$, we see that this leaves us with a quasi-cubic graph whose univalent vertices are precisely those vertices $v_1,\dots,v_n$ attached to $e_1,\dots,e_n$.  Hence, for any minimal cut set $e_1,\dots,e_n$ of $G_e$, the sum of colours assigned to $e_1,\dots,e_n$ must be zero.\medskip
\end{defn}

\begin{defn}[Star]A star is a set of 5 vertices of $H$, $u_1,\dots,u_5$ together with 5 edges such that $(u_i, u_{i+2})$, $(u_i, u_{i+3})$ are edges, where $+$ denotes addition in $\mathbb{Z}_5$. Let $v_1,\dots,v_5$ be vertices (necessarily distinct from the $u_i$'s) such that $(u_i,v_i)$ are edges of $H$.\medskip
\end{defn}

\begin{thm}[Star Colouring Theorem] If $H$ is coloured with the set of three colours $C$:=$\{a,b,c\}$, then for some distinct choices $x,y,z$ of $\{a,b,c\}$, and some $i \in {1,\dots,5}$, then the colours of $(u_i,v_i),\dots,(u_{i+4},v_{i+4})$ are $x,y,x,x,z$ (respectively).\medskip
\end{thm}

\begin{proof}By the Parity Lemma applied to the minimal cut set $(u_i,v_i),\dots,(u_{i+4},v_{i+4})$, these edges must be coloured so that 3 edges have one colour, and the two other edges have unique colours. Hence, for some distinct colours $x,y,z$ of C, 3 edges must be coloured $x$, 1 edge must be coloured $y$, and 1 must be coloured $z$.  Furthermore, the 3 edges coloured $x$ cannot be consecutive (ie: of the form $(u_i,v_i)$, $(u_{i+1},v_{i+1})$, $(u_{i+2},v_{i+2})$), since this would lead to the adjacent edges $(u_i,u_{i+3})$ and $(u_{i+1},u_{i+3})$ receiving the same colour. Theorem 2.11 follows.\medskip
\end{proof}

Note that a colouring of the "outer edges" $(u_i,v_i),\dots,(u_{i+4},v_{i+4})$ of the star in this manner determines the colouring of the "inner edges" $(u_i, u_{i+2})$, $(u_i, u_{i+3})$.\medskip

\section{The Numbers $\psi(G,e)$}

The central question for this note is: for what numbers $n$ do there exist a snark $G$ containing an edge $e$ such that $\psi(G,e)=n$?  Our method of answering the question entails using well-known methods of building "bigger" snarks from "smaller" snarks in ways such that the number $\psi$ for the "big" snark and some edge of it is related in a specified way to the numbers $\psi$ for the smaller snarks and respective edges therein.  We can then proceed by an inductive process, starting with the Peterson graph, to build infinite families of snarks with respective edges for which we can determine the numbers $\psi$.\medskip

\section{The Symmetric Dot Product}
In \cite{Br}, Bradley gave a proof of Theorem 4.1.  His argument relies on a method for combining two snarks into a larger snark, the symmetric dot product, a special case of Isaacs' \cite{Is} general dot product.\medskip

\begin{thm}Let $i,j$ be nonnegative integers (possibly 0).  There exists a snark $G$ with an edge $e$ such that $\psi (G,e) = 2^i \cdot 3^j$.\medskip
\end{thm}

\begin{defn}[The Symmetric Dot Product]Let $G'$ and $\hat{G}$ be two snarks.  Suppose $E=(U,V)$ is an edge of $G'$, and that $U_1,U_2,V_1,V_2$ are distinct vertices of $G'$ such that $(U_i,U), (V_i,V), \{i=1,2\}$ are edges of $G'$. Similarly, suppose $\epsilon=(u,v)$ is an edge of $\hat{G}$, and that $u_1,u_2,v_1,v_2$ are distinct vertices of $\hat{G}$ such that $(u_i,u), (v_i,v), \{i=1,2\}$ are edges of $\hat{G}$.\medskip

Let $G_0':=G'-\{U,V\}$ and $\hat{G}_0:=\hat{G}-\{u,v\}$.\medskip

Let G denote the simple, cubic graph consisting of $G_0'$, $\hat{G}_0$, two new vertices $T$ and $W$,  and the seven new edges $\omega$:=$(T,W)$, $d_1$:=$(U_1,u_1)$, $d'_2$:=$(U_2,T)$, $\hat{d}_2$:=$(T,u_2)$, $D'_1$:=$(V_1,W)$, $\hat{D}_1$:=$(W,v_1)$, and $D_2$:=$(V_2,v_2)$.\medskip

It is well-known that $G$ is a snark.  See, for example, \cite{Is}.\medskip
\end{defn}
 
\begin{thm}If $e$ is an edge of $\hat{G}_0$, then $e$ is also an edge of $G$, and $\psi(G,e)=2 \cdot \psi(\hat{G},e) \cdot \psi(G',E)$.\medskip
\end{thm}

\begin{thm}In the above context, $\psi (G, d_1) = 3 \cdot \psi (G',(u,u_1))\cdot \psi (\hat{G},(U_1,U))$.\medskip
\end{thm}

\begin{remark}The reader is referred to \cite{Br} for full proofs. One can use the fact that $\psi (P,p)=1$, where $P$ is the Peterson Graph and $p$ is any edge of the Peterson graph, to induct on the numbers $\psi(G,e)$.  Theorem 4.1 follows.\medskip
\end{remark}

\section{Superposition}

There is a second method of "hooking" snarks together to obtain new snarks: superposition.  We use a special case of a much more general method devised by Kochol \cite{Ko}.  This case of superposition involves hooking two Peterson graphs $P_1$ and $P_2$ together, then piecing them to a general snark $G_0$ to obtain a "big" snark $G$.  Theorems 5.1 and 5.3 give relationships between the numbers $\psi$ for certain edges in the original snark $G_0$ and the snark obtained by superposition, $G$.\medskip

\begin{defn}[Superposition]

We begin by hooking two Peterson graphs together as depicted in the following:\medskip

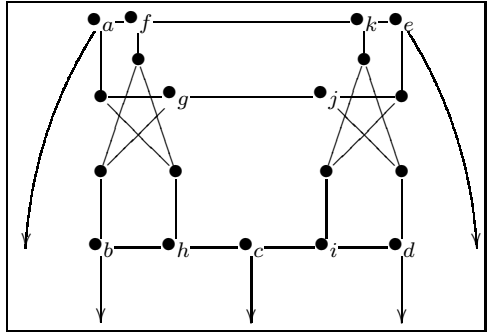
\objectmargin{0pt}
\[ \fbox{ \xygraph{ !{<0cm,0cm>;<1cm,0cm>:<0cm,1cm>::}
!{(0,1)}*+{}="P_1" !{(1,0)}*+{}="P_2" !{(3,0)}*+{}="P_3"
!{(5,0)}*+{}="P_4" !{(6,1)}*+{}="P_5" !{(1,1)}*+{\bullet_b}="a"
!{(2,1)}*+{\bullet_h}="b" !{(3,1)}*+{\bullet_c}="c"
!{(4,1)}*+{\bullet_i}="d" !{(5,1)}*+{\bullet_d}="e"
!{(1,2)}*+{\bullet}="f" !{(2,2)}*+{\bullet}="g"
!{(4,2)}*+{\bullet}="h" !{(5,2)}*+{\bullet}="i"
!{(1,3)}*+{\bullet}="j" !{(2,3)}*+{\bullet_g}="k"
!{(4,3)}*+{\bullet_j}="l" !{(5,3)}*+{\bullet}="m"
!{(1.5,3.5)}*+{\bullet}="x" !{(4.5,3.5)}*+{\bullet}="y"
!{(1,4)}*+{\bullet_a}="n" !{(1.5,4)}*+{\bullet_f}="o"
!{(4.5,4)}*+{\bullet_k}="p" !{(5,4)}*+{\bullet_e}="q"
"a"-@{->}"P_2"
"a"-"b" "b"-"c" "c"-@{->}"P_3" "c"-"d" "d"-"e" "e"-@{->}"P_4"
"a"-"f" "b"-"g" "d"-"h" "e"-"i" "f"-"x" "x"-"g" "g"-"j" "j"-"k"
"k"-"f" "k"-"l" "l"-"m" "m"-"h" "h"-"y" "y"-"i" "i"-"l" "n"-@/_/@{->}"P_1"
"n"-"o" "o"-"p" "p"-"q" "q"-@/^/@{->}"P_5" "n"-"j" "x"-"o" "y"-"p" "m"-"q"
} }  \]\medskip

The vertex-free edges connected to vertices $a,\dots,e$ will be joined to the snark $G_0$, as follows.\bigskip

Choose a path of five distinct vertices of $G_0$, $\{u_1,\dots,u_5\}$, so that $(u_1,u_2),\dots,(u_4,u_5)$ are edges of $G_0$. Let $\{v_2,v_3,v_4\}$ be vertices of $G_0$ (distinct from the $u_i$'s such that $(u_2,v_2)$, $(u_3,v_3)$, $(u_4,v_4)$ are edges of $G_0$ (the $v_i$'s are necessarily distinct from the $u_i$'s; otherwise, there would be a cycle of 4 or fewer edges in $G_0$).\medskip

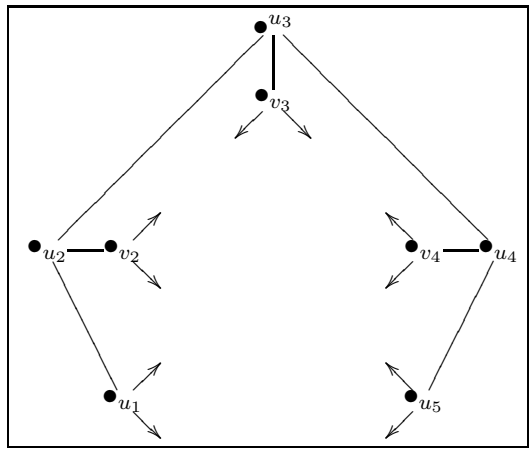
\objectmargin{0pt}
\[ \fbox{ \xygraph{ !{<0cm,0cm>;<1cm,0cm>:<0cm,1cm>::}
!{(1,1)}*+{\bullet_{u_1}}="u_1" !{(1.5,1.5)}*+{}="v_1"
!{(1.5,.5)}*+{}="a" !{(0,3)}*+{\bullet_{u_2}}="u_2"
!{(1,3)}*+{\bullet_{v_2}}="v_2" !{(1.5,3.5)}*+{}="c"
!{(1.5,2.5)}*+{}="b" !{(3,6)}*+{\bullet^{u_3}}="u_3"
!{(3,5)}*+{\bullet_{v_3}}="v_3" !{(2.5,4.5)}*+{}="e"
!{(3.5,4.5)}*+{}="d" !{(6,3)}*+{\bullet_{u_4}}="u_4"
!{(5,3)}*+{\bullet_{v_4}}="v_4" !{(4.5,3.5)}*+{}="f"
!{(4.5,2.5)}*+{}="g" !{(5,1)}*+{\bullet_{u_5}}="u_5"
!{(4.5,1.5)}*+{}="v_5" !{(4.5,.5)}*+{}="i" "u_1"-@{->}"v_1"
"u_1"-@{->}"a" "u_1"-"u_2" "u_2"-"v_2" "v_2"-@{->}"b" "v_2"-@{->}"c"
"u_2"-"u_3" "u_3"-"v_3" "v_3"-@{->}"d" "v_3"-@{->}"e" "u_3"-"u_4"
"u_4"-"v_4" "v_4"-@{->}"f" "v_4"-@{->}"g" "u_4"-"u_5" "u_5"-@{->}"i"
"u_5"-@{->}"v_5" }} \]\medskip

We extend $G_0$ by eliminating the edges $(u_2,u_3)$, $(u_3,u_4)$, and replacing edges $(u_1,u_2)$, $(u_2,v_2)$, $(u_3,v_3)$, $(u_4,v_4)$, $(u_4,u_5)$ with $(u_1,a)$, $(v_2,b)$, $(v_3,c)$, $(v_4,d)$, $(u_5,e)$ (respectively).  This yields $G$.
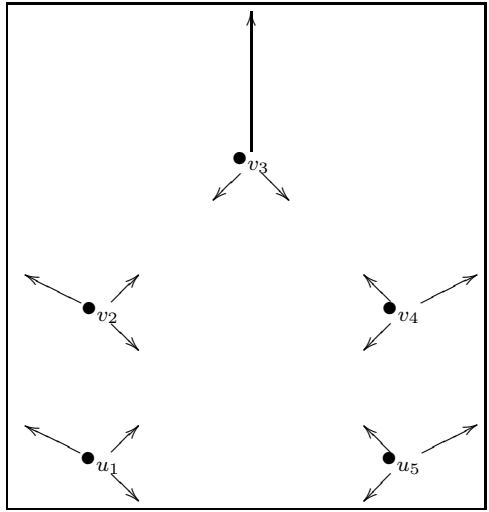
\medskip

\objectmargin{0pt}
\[ \fbox{ \xygraph{ !{<0cm,0cm>;<1cm,0cm>:<0cm,1cm>::}
!{(1,1)}*+{\bullet_{u_1}}="u_1" !{(1.5,1.5)}*+{}="v_1"
!{(0,1.5)}*+{}="w_1" !{(1.5,.5)}*+{}="a" !{(0,3.5)}*+{}="u_2"
!{(1,3)}*+{\bullet_{v_2}}="v_2" !{(1.5,3.5)}*+{}="c"
!{(1.5,2.5)}*+{}="b" !{(3,7)}*+{}="u_3"
!{(3,5)}*+{\bullet_{v_3}}="v_3" !{(2.5,4.5)}*+{}="e"
!{(3.5,4.5)}*+{}="d" !{(6,3.5)}*+{}="u_4"
!{(5,3)}*+{\bullet_{v_4}}="v_4" !{(4.5,3.5)}*+{}="f"
!{(4.5,2.5)}*+{}="g" !{(5,1)}*+{\bullet_{u_5}}="u_5"
!{(4.5,1.5)}*+{}="v_5" !{(4.5,.5)}*+{}="i" !{(6,1.5)}*+{}="w_5"
"u_1"-@{->}"v_1" "u_1"-@{->}"w_1" "u_1"-@{->}"a" "v_2"-@{->}"u_2"
"v_2"-@{->}"b" "v_2"-@{->}"c" "v_3"-@{->}"u_3" "v_3"-@{->}"d"
"v_3"-@{->}"e" "v_4"-@{->}"u_4" "v_4"-@{->}"f" "v_4"-@{->}"g"
"u_5"-@{->}"i" "u_5"-@{->}"v_5" "u_5"-@{->}"w_5"  } } \]

It is well known that $G$, pictured in full as follows, is a snark. See e.g. \cite{Ko}). The addendum includes a proof of the fact that the "superposition" operation preserves the cyclic 5-edge-connectededness of the original snark $G_0$.\medskip

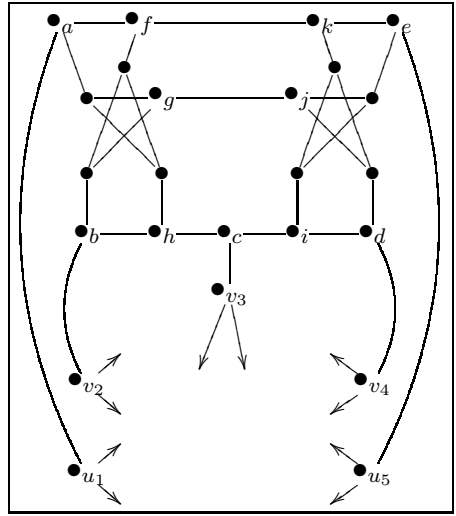
\objectmargin{0pt}
\[ \fbox{ \xygraph{ !{<0cm,0cm>;<1cm,0cm>:<0cm,2cm>::}
!{(1.1,1)}*+{\bullet_{u_1}}="u_1" !{(1.55,1.2)}*+{}="v_1"
!{(1.55,.8)}*+{}="w_1" !{(1.1,1.6)}*+{\bullet_{v_2}}="v_2" !{(1.55,1.8)}*+{}="u_2" !{(1.55,1.4)}*+{}="w_2"
!{(3,2.2)}*+{\bullet_{v_3}}="v_3"
!{(2.6,1.7)}*+{}="u_3" !{(3.2,1.7)}*+{}="w_3"
!{(4.9,1.6)}*+{\bullet_{v_4}}="v_4"
!{(4.35,1.8)}*+{}="u_4" !{(4.35,1.4)}*+{}="w_4"
!{(4.9,1)}*+{\bullet_{u_5}}="u_5"
!{(4.35,1.2)}*+{}="v_5"
!{(4.35,.8)}*+{}="w_5"
!{(1.1,2.6)}*+{\bullet_b}="a"
!{(2.1,2.6)}*+{\bullet_h}="b" !{(3,2.6)}*+{\bullet_c}="c"
!{(3.9,2.6)}*+{\bullet_i}="d" !{(4.9,2.6)}*+{\bullet_d}="e"
!{(1.1,3)}*+{\bullet}="f" !{(2.1,3)}*+{\bullet}="g"
!{(3.9,3)}*+{\bullet}="h" !{(4.9,3)}*+{\bullet}="i"
!{(1.1,3.5)}*+{\bullet}="j" !{(2.1,3.5)}*+{\bullet_g}="k"
!{(3.9,3.5)}*+{\bullet_j}="l" !{(4.9,3.5)}*+{\bullet}="m"
!{(1.6,3.7)}*+{\bullet}="x" !{(4.4,3.7)}*+{\bullet}="y"
!{(.75,4)}*+{\bullet_a}="n" !{(1.8,4)}*+{\bullet_f}="o"
!{(4.2,4)}*+{\bullet_k}="p" !{(5.25,4)}*+{\bullet_e}="q"
"a"-"b" "b"-"c" "c"-"d" "d"-"e" "a"-"f" "b"-"g"
"d"-"h" "e"-"i" "f"-"x" "x"-"g" "g"-"j" "j"-"k" "k"-"f" "k"-"l" "l"-"m" "m"-"h" "h"-"y" "y"-"i" "i"-"l" "n"-"o" "o"-"p" "p"-"q" "n"-"j" "x"-"o" "y"-"p"
"m"-"q" "c"-"v_3" "v_2"-@/^3mm/"a" "u_1"-@/^7mm/"n" "v_4"-@/_3mm/"e" "u_5"-@/_7mm/"q"
"u_1"-@{->}"v_1" "u_1"-@{->}"w_1"
"v_2"-@{->}"u_2" "v_2"-@{->}"w_2"
"v_3"-@{->}"u_3" "v_3"-@{->}"w_3"
"v_4"-@{->}"u_4" "v_4"-@{->}"w_4"
"u_5"-@{->}"v_5" "u_5"-@{->}"w_5"
} }  \]\medskip

\end{defn}

\begin{thm}Let E:=$(v_3,c)$, and $\epsilon:=(u_3,v_3)$.  Then $\psi(G,E)=7 \cdot \psi(G_0,\epsilon)$.\medskip
\end{thm}

\begin{lem}In any colouring of $G_E$, the edges $(u_1,a)$ and $(v_2,b)$ are assigned different colours.\medskip
\end{lem}

\begin{proof}By the Parity Theorem applied to the minimal cut set consisting of $(u_1,a)$, $(v_2,b)$, $(h,i)$, $(g,j)$, $(f,k)$, we see that 3 of these 5 edges must share the same colour, and the other 2 must have unique colours.  Suppose $(u_1,a)$ and $(v_2,b)$ share the same colour.  Then $(h,i)$, $(g,j)$, $(f,k)$ are coloured distinctly. This leads to a colouring of $P$, the Peterson graph; contradiction. Lemma 5.2 follows.\medskip
\end{proof}

\begin{proof}[Proof of Theorem 5.1]

We will show that a 3-edge-decomposition of $(G_0)_\epsilon$ induces precisely 7 possible 3-edge-decompositions of $G_E$.\medskip

In order to count the 3-edge-decompositions of $G_E$, we fix the colours for edges $(u_1,a)$, $(v_2,b)$ as $x$ and $y$, respectively. For each 3-edge-decomposition of $G_E$, there will be exactly one such colouring of $G_E$ meeting that specification.  Likewise, in order to count the 3-edge-decompositions of $(G_0)_\epsilon$, we fix the colours of edges $(u_1,u_2)$ and $(u_2,v_2)$ as $x$ and $y$.\medskip

Consider such a colouring $f$ of $(G_0)_\epsilon$.  Since $f((u_1,u_2))=x$ and $f((u_2,v_2))=y$, we have $f((u_2,u_4))=z$, where z is the third colour of $C$.  Furthermore, $(u_4,v_4)$ and $(u_4,u_5)$ are coloured $x$ and $y$ (in some order), since they are mutually adjacent to an edge coloured $z$ (the edge $(u_2,u_4)$).\medskip

Now, $f$ induces a number of colourings $f_1, \dots, f_n$ of $G_E$ such that, for each $i\in\{1,\dots, n\}$:\medskip

$f_i(u_1, a)=f(u_1, u_2)=x$\smallskip
$f_i(v_2, b)=f(u_2, v_2)=y$ (as in the initial constraints)\smallskip
$f_i(v_4, d)=f(u_4, v_4)$\smallskip
$f_i(u_5, e)=f(u_4, u_5)$\smallskip
and $f_i(d)=f(d)$ for all edges d common to both $(G_0)_\epsilon$ and $G_E$.\smallskip

Note that, for each $f_i$, the minimal cut set $(u_1,a)$,$(v_2,b)$,$(v_4,d)$,$(v_5,e)$ is coloured either $x$, $y$, $x$, $y$ or $x$, $y$, $y$, $x$, so the sums of these edges are 0.\medskip

It is not difficult to see, using the Star Colouring Theorem applied to the "stars" arising from the original two Peterson graphs used in the construction, that each colouring $f_i$ is completely determined by the colourings for the three "top" edges $(h,i)$, $(g,j)$, $(f,k)$.\medskip

Keeping in mind that these three edges $(h,i)$, $(g,j)$, $(f,k)$, together with the edges $(u_1,a)$, $(v_2,b)$, form a minimal cut set of $G_E$, we see that\medskip

$f_i((u_1,a))+f_i((v_2,b))+f_i((e_1))+f_i((e_2))+f_i((e_3))=0$\medskip

Since $f_i((u_1,a))=x$ and $f_i((v_2,b))=y$, by assumption, it must be the case that $f_i((e_1))+f_i((e_2))+f_i((e_3))=z$.\medskip

There are precisely 7 colourings for the three edges $e_1, e_2, e_3$ satisfying this last constraint, namely those with $f_i(e_1 e_2 e_3)$=$xxz$, $xzx$, $zxx$, $yyz$, $yzy$, $zyy$, $zzz$. \medskip

It follows that each colouring $f$ of $(G_0)_\epsilon$ induces 7 colourings $f_1, \dots, f_7$ of $G_E$.\medskip

Furthermore, by an argument involving the Parity Lemma applied to the cut set consisting of edges $(u_1,a), (v_2,b), (v_4,d), (v_5,e)$, any colouring of $G_E$ meeting the specification given by Lemma 5.2 (with colours of $(u_1,a), (v_2,b)$ fixed as $x$ and $y$) will induce a unique colouring of $(G_0)_\epsilon$.  Hence, there will be no colourings of $G_E$ other than the ones induced by the colourings of $(G_0)_\epsilon$ (as above).\medskip

Theorem 5.1 follows.\medskip

\end{proof}

\begin{thm}Let F be an edge of $G_0$ other than one of the seven edges $(u_i,u_i+1)$, $i\in\{1,2,3,4\}$ or $(u_i,v_i)$, $i \in \{2,3,4\}$. Then $\psi(G,F)=5\cdot \psi(G_0,F)$\medskip
\end{thm}

\begin{lem}In any colouring of $G_F$, the edges $(u_1,a)$ and $(v_2,b)$ are assigned different colours.\medskip
\end{lem}
\begin{proof}Similar to that of Lemma 5.2.\medskip
\end{proof}

\begin{proof}[Proof of Theorem 5.3]
In order to count the 3-edge-decompositions of $G_F$, we fix the colours for edges $(u_1,a)$, $(v_2,b)$ as x and y.  For each 3-edge-decomposition of $G_F$, there will be exactly one such colouring of $G_F$ meeting that specification.  Similarly, in order to count the 3-edge-decompositions of$(G_0)_F$ the colours of edges $(u_1,u_2)$ and $(u_2,v_2)$ will be fixed as x and y.\medskip

We will show that each colouring $f$ of $(G_0)_F$ satisfying the above constraint induces precisely 5 colourings of $G_F$.\medskip

Consider such a colouring $f$ of $(G_0)_F$.  Since $f((u_1,u_2))=x$ and $f((u_2,v_2))=y$, we have $f((u_2,u_3))=z$, where $z$ is the third colour of C, so that $(u_3,v_3),(u_3,u_4)$ are coloured (resp.) $x, y$ ['Case A'] or $y, x$ ['Case B'], and that $(u_4,v_4), (u_4,u_5)$ are coloured (resp.) $x, z$ or $z, x$ [in Case A] or $y, z$ or $z, y$ [in Case B].\medskip

Now, $f$ induces a number of colourings $f_1,\dots,f_n$ of $G_F$ such that, for each $i \in \{1,\dots,n\}$:\medskip

$f_i((u_1,a))=f((u_1,u_2))=x$, $f_i((v_2,b))=f((u_2,v_2))=y$, $f_i((v_3,c))=f((u_3,v_3))$, $f_i((v_4,d))=f((u_4,v_4))$, $f_i((u_5,e))=f((u_4,u_5))$, and $f_i(d)=f(d)$ for all edges d common to both $G_F$ and $(G_0)_F$.\medskip

Note that, in each case, the colours given to the minimal cut set $(u_1,a)$, $(v_2,b)$, $(v_3,c)$, $(v_4,d)$, $(u_5,e)$ sum to 0.\medskip

Likewise, by applying the Parity Lemma to the minimal cut set $(u_1,a)$, $(v_2,b)$, $(c,h)$, $(g,j)$ and $(k,f)$, we see that $f_i(u_1,a)+f_i(v_2,b)+f_i(c,h)+f_i(g,j)+f_i(k,f)=0$, so that $f_i(c,h)+f_i(g,j)+f_i(k,f)=z$.\medskip

As in the proof of Theorem 5.1, there are precisely 7 possibilities for colourings of $(c,h), (g,j), (k,f)$ satisfying this constraint: $xxz$, $xzx$, $zxx$, $yyz$, $zyz$, $zzy$, $zzz$. However, for each colouring $f_i$, edge $(v_3,c)$ is coloured either $x$ or $y$ (Case A or Case B, respectively), so edge $(c,h)$ cannot be coloured (resp.) $x$ or $y$. In each case, only 5 of the 7 colouring possibilities for $(c,h)$, $(g,j)$, $(k,f)$ satisfy this constraint.\medskip

Each of these choices of colourings of the minimal cut set $(u_1,a)$, $(v_2,b)$, $(c,h)$, $(g,j)$ and $(k,f)$, lead to unique colourings of the remnants of the top two Peterson graphs and the remainder of $G_F$.  Hence, each colouring $f$ of $(G_0)_F$ induces 5 colourings $f_1,\dots, f_5$ of $G_0$.\medskip

Furthermore, by an argument involving the Parity Lemma applied to the minimal cut set consisting of edges $(u_1,a)$, $(v_2,b)$, $(v_3,c)$, $(v_4,d)$, $(u_5,e)$, each colouring of $G_F$ satisfying the specification in Lemma 5.4 leads to precisely one colouring of $(G_0)_F$ - with edges $(u_1,u_2)$, $(u_2,v_2)$, $(u_3,v_3)$, $(u_4,v_4)$, and $(u_4,u_5)$ assigned the colours of (respectively) $(u_1,a)$, $(v_2,b)$, $(v_3,c)$, $(v_4,d)$, $(u_5,e)$.  Hence, there will be no colourings of $G_F$ other than the ones induced by the colourings of $(G_0)_F$ (as above).\medskip

Theorem 5.3 follows.
\end{proof}

\begin{thm}Let $i, j$ be nonnegative integers, and let $n=5^i \cdot 7^j$. There exists a cyclically 5-edge-connected snark $G$ with an edge $e$ such that $\psi(G,e)=n$.
\end{thm}

\begin{proof}Using (as a starting point) the fact that $\psi(P,p)=1$, where $P$ is the Peterson graph and $p$ is any edge therein, we can induct using repeated applications of Theorem 5.1 and 5.3 to obtain Theorem 5.5.  At each step of the induction, the cyclic 5-edge-connectedness is preserved (by Lemma A7 of the Appendix).
\end{proof}

\begin{cor}Let $i, j, k, l$ be nonnegative integers, and let $n=2^i \cdot 3^j \cdot 5^k \cdot 7^l$.  Then there exists a (ordinary, at least cyclically 4-edge-connected) snark $G$ with an edge $e$ such that $\psi(G,e)=n$.\medskip
\end{cor}

\begin{proof}Using Theorem 5.5 as a starting point, we can use repeated applications of Theorems 4.2, 4.3, to establish Corollary 5.6.  At each step of the induction, the cyclic 4-edge-connectedness is preserved (by Remark A8 of the Appendix).\medskip
\end{proof}

\section{Future Work}
 The special case of superposition used in the arguments for Theorems 5.1 and 5.3 does not lead to any other "new" numbers $\psi(G,e)$: for each edge $e$ of the "big" snark $G$, the number $\psi(G,e)$ only has $2, 3, 5,$ or $7$ as prime factors.  Hence, any new possibilities for the numbers $\psi(G,e)$ using a similar inductive approach will require a different method of constructing snarks.  The "flower snark" construction given by Isaacs \cite{Is} and "BlowUp" construction given by \cite{Ha} seem to be possibly promising methods for further investigation.\medskip

\section{Acknowledgements}
I would like to express my gratitude to Professor Richard Bradley for his suggestion of this research topic as a project in the 2006 Research Experiences for Undergraduates (REU) program at Indiana University, as well as his provision of the appendix regarding "cyclic edge-connectedness" included below.  I would also like to acknowledge the US National Science Foundation (NSF) for their sponsorship of the REU program.

\vfill\eject

\section{Appendix}

\centerline {\bf Notes on cyclic edge-connectedness}
\bigskip

This material is taken essentially verbatim from notes of Bradley
\cite{RB}, with his encouragement.
\bigskip

In connection with Kochol's \cite{Ko} ``superposition''
technique for creating ``big'' snarks from ``smaller'' ones, there
is a body of information on ``cyclic edge-connectedness'' that is
well established in extensive generality and appears to be simply
taken for granted by the snark specialists, and for which a
convenient reference seems hard to find. In what follows below,
enough of that known background information will be supplied (in a
convenient though very restricted form) to suffice for the
verification that a certain ``cyclic 5-edge-connectedness'' property
is ``preserved'' in a pertinent particular application of Kochol's
technique.
\bigskip

   Suppose $G$ is a simple graph (that is, an undirected,
connected graph with finitely many vertices and no loops or multiple
edges).  Two subgraphs $A$ and $B$ of $G$ are said to be
``edge-disjoint'' if they have no edges in common (they may have one
or more vertices in common), or simply ``disjoint'' if they have no
vertices in common.
\bigskip

   If $G$ is a simple graph and $S$ is a (nonempty) set of some
edges of $G$, then in what follows, $G-S$ will denote the graph that
one obtains from $G$ by deleting all edges in $S$. (No vertices are
deleted; the graph $G-S$ retains all vertices of $G$, including the
ones that were endpoints of the edges in $S$.)

   If $G$ is a simple graph and $S$ is a set of some edges of $G$,
then $S$ is a ``cut set'' if the graph $G-S$ is disconnected.
\bigskip

   Suppose (say) $G$ is a simple graph, and $A$ and $B$ are
disjoint subgraphs of $G$, with $A$ and $B$ each being connected.
These two graphs $A$ and $B$ are ``5-edge-connected to each other in
$G$'' if either one (hence both) of the following equivalent
conditions holds:

   (i) There exist five mutually edge-disjoint paths in $G$,
each of which has one endpoint in $A$ and the other endpoint in $B$.

   (ii) There does not exist a cut set $S$ consisting of four
or fewer edges of $G$ such that $A$ and $B$ are respectively
entirely in (i.e. subgraphs of) different components of $G-S$.

   The equivalence of (i) and (ii) is a special case of the
``edge form'' of Menger's Theorem.  See e.g. (\cite{GY}, p.\ 196,
Theorem 5.3.10). In that formulation of the ``edge form'' of
Menger's Theorem, the graph is not assumed to be simple
--- it can have ``multiple edges.''  To apply that formulation here,
first contract each of the sets $A$ and $B$ to a vertex.
\bigskip

   Now let us turn to cycles (``closed paths''). \bigskip

   {\bf Lemma A1.} \quad {\sl Suppose $G$ is a simple graph such
that every vertex in $G$ has valence at least 2.  Then $G$ contains
at least one cycle.}
\bigskip

   This is a well known elementary fact.  To obtain a cycle, one
starts at some vertex, chooses a path ``at random,'' one edge at a
time (with no ``backtracking''), until one hits a vertex that has
already been hit previously.
\bigskip

   A simple graph $G$ is ``cyclically 5-edge-connected'' if every
two disjoint cycles in $G$ are 5-edge-connected to each other in
$G$.  (This definition is considered to hold ``by default'' for any
simple graph that does not have two disjoint cycles.  In what
follows below, such ``trivial'' cases will not occur.)

   A ``5-cycle'' is a cycle with exactly 5 edges.

   A ``hinge'' is a graph consisting of two adjacent edges and
their endpoints --- that is, three vertices $v_1, v_2, v_3$ and the
two edges $(v_1, v_2)$ and $(v_2, v_3)$.
\bigskip

   {\bf Lemma A2.} \quad {\sl Suppose $G$ is a simple cubic graph
which is cyclically 5-edge- \hfil\break connected, and $Q$ is a
5-cycle in $G$. Suppose $H$ is a hinge in $G$ which is disjoint from
$Q$.  Then $Q$ and $H$ are 5-edge-connected to each other in $G$.}
\bigskip

   {\it Proof.} \quad Suppose instead that there exists a
cut set $S$ of four or fewer edges in $G$ such that $H$ and $Q$ are
in different components of $G-S$.  Choose such a set $S$ with
minimal cardinality.  Let $F$ denote the component of $G-S$ that
contains $H$.  We shall use Lemma A1 to show that $F$ has a cycle.
Thereby a contradiction will have occurred, and Lemma A2 will be
proved.

   If no vertex in $F$ is attached (in $G$) to more than one edge
in $S$, then every vertex in $F$ has valence at least 2 (within $F$
itself), and Lemma A1 applies simply to $F$.

   If exactly one vertex $v$ in $F$ is attached to two edges in
$S$, then let $e$ denote the third edge (the one in $F$) attached to
$v$, and apply Lemma A1 to $F-\{e,v\}$.  (The notation $F-\{e,v\}$
means the graph that one obtains from $F$ by deleting the edge $e$
and the vertex $v$.  Similar notation is used in the next
paragraph.)

   The only remaining case is where exactly two different
vertices $u$ and $v$ in $F$ are each attached to two edges in $S$.
Then there is not an edge $(u,v)$ (for otherwise that edge by itself
would be a component of $G-S$ and it would have to be $F$,
contradicting the stipulation that the hinge $H$ is in $F$). Let
$e_1$ resp.\ $e_2$ denote the edge in $F$ attached to the vertex $u$
resp.\ $v$.  If $e_1$ and $e_2$ are not adjacent, then apply Lemma
A1 to $F - \{e_1,e_2,u,v\}$.  If instead $e_1$ and $e_2$ are
adjacent, attached to the same vertex $w$, then let $e_3$ denote the
third edge attached to $w$ ($e_3$ will be in $F$, not in $S$), and
apply Lemma A1 to $F - \{e_1, e_2, e_3, u,v,w\}$.  That finishes the
proof.
\bigskip

   {\bf Construction A3.} \quad Now let us look at one particular
operation.   It will be repeated here in a form convenient for the
rest of this Appendix.  This construction, and its properties
mentioned below, are (at least in principle) already known as a
special case of the general ``superposition'' techniques of Kochol.
\bigskip

   (1a) Let $P_1$ and $P_2$ be two disjoint Petersen graphs.

   (1b) Let $(u_1, u_2, u_3, u_4, u_5)$ be, in order, the five
vertices of a particular 5-cycle in $P_1$.  Let $u_2'$ denote the
other vertex (besides $u_1$ and $u_3$) that is adjacent to $u_2$.

   (1c) Let $(v_1, v_2, v_3, v_4, v_5)$ be, in order, the five
vertices of a particular 5-cycle in $P_2$. Let $v_2'$ denote the
other vertex (besides $v_1$ and $v_3$) that is adjacent to $v_2$.

   (1d) Let $P_d$ (the subscript $d$ stands for ''double'') denote
the simple cubic graph that one obtains from $P_1$ and $P_2$ by (i)
deleting from $P_1$ the vertex $u_2$ and all three edges attached to
it, (ii) deleting from $P_2$ the vertex $v_2$ and all three edges
attached to it, and then (iii) inserting three new edges $(u_1,
v_1)$, $(u_2', v_2')$, and $(u_3,v_3)$.

   (1e) Every cycle in $P_d$ has at least 5 edges.  For any
such cycle that was originally in $P_1$ or $P_2$, this holds
trivially since the Petersen graph has girth 5.  For any cycle $C$
in $P_d$ that has at least one (and hence exactly two) of the three
``new edges'' in (1d)(iii), the number of edges in $C$ is in fact at
least 8.

   (1f) Let $\tilde P_d$ denote the graph that one obtains from
$P_d$ by (i) deleting the edges $(u_4, u_5)$ and $(v_4,v_5)$ (but
leaving their endpoints intact), and (ii) inserting a vertex $\nu_3$
on the edge $(u_3,v_3)$ (thereby replacing the edge $(u_3,v_3)$ by
the two new edges $(u_3, \nu_3)$ and $(\nu_3, v_3)$).

   (1g) This graph $\tilde P_d$ has exactly five 2-valent vertices:
$u_5, u_4, \nu_3, v_4$, and $v_5$.  Also, no edge in $\tilde P_d$ is
attached to more than one of those five vertices. (All other
vertices in $\tilde P_d$ are 3-valent.)

   (1h) As a trivial consequence of (1e), every cycle in
$\tilde P_d$ has at least 5 edges.
\medskip

   (2a) Now suppose $G_0$ is a cyclically 5-edge-connected
snark that is disjoint from the graphs $P_1$, $P_2$, $P_d$, $\tilde
P_d$ in (1a)-(1h). Suppose $H_0$ is a hinge in $G_0$.  Suppose $Q_0$
is a 5-cycle in $G_0$ such that $Q_0$ is disjoint from $H_0$.

   (2b) Let the vertices of $H_0$ be denoted $t_1, t_2, t_3$, with
$H_0$ having edges $(t_1, t_2)$ and $(t_2,t_3)$. Let the five edges
attached to $H_0$ (but not in $H_0$) be denoted by $e_1 := (t_1,
w_1)$, $e_2 := (t_1, w_2)$, $e_3 := (t_2, w_3)$, $e_4 := (t_3,
w_4)$, $e_5 := (t_3, w_5)$. (That is, the ``other endpoints'' of
those edges are $w_1, \dots, w_5$.)\ \ Since $G_0$ (a snark with a
5-cycle $Q_0$) has girth 5, it follows that the vertices $w_i$ are
distinct.

   (2c) Let $\tilde G_0$ denote the graph that one obtains from
$G_0$ by deleting the hinge $H_0$ (its vertices and edges) as well
as the five edges $e_i$ connected to $H_0$.  Then $\tilde G_0$ has
five 2-valent vertices $w_1, \dots, w_5$; its other vertices are
each 3-valent. The 5-cycle $Q_0$ is a subgraph of $\tilde G_0$.
(Some of the $w_i$'s may be vertices of $Q_0$.)

   (2d) Every cycle in $\tilde G_0$ has at least 5 edges
(since every cycle in the original snark $G_0$ has at least 5
edges.)
\bigskip

   (3a) Finally, let $G$ denote the cubic graph that one obtains
from $\tilde P_d$ and $\tilde G_0$ by inserting the five new edges
$E_1 := (u_5, w_1)$, $E_2 := (u_4, w_2)$, $E_3 := (\nu_3, w_3)$,
$E_4 := (v_4, w_4)$, and $E_5 := (v_5, w_5)$.

   (3b) If a given cycle $C$ in $G$ has at least one (and hence
exactly two or four) of the five new edges $E_1, \dots, E_5$, then
$C$ has at least 5 edges, and $C$ also has at least one hinge (i.e.\
two adjacent edges) in $\tilde P_d$. (Those facts are trivial
consequences of (1g).)

   (3c) Every cycle $C$ in the graph $G$ has at least 5 edges
(by (1h), (2d), and (3b)).

   (3d) If $C$ is a cycle in $G$, then either $C$ is entirely
in $\tilde G_0$ or $C$ has a hinge in $\tilde P_d$. (See (3b)
again).

   (3e) In $G$, one can trivially pick out three 5-cycles
that are mutually disjoint from each other (for example, one being
$Q_0$, and the other two being originally in $P_1$ and $P_2$
respectively).

   (3f) If $H$ is any hinge in $G$, then $G$ has a 5-cycle
$Q^*$ that is disjoint from $H$.  (By a careful but simple argument,
it is not possible for all three of the disjoint 5-cycles mentioned
in (3e) above to contain a vertex of $H$.)
\bigskip

   The remaining four lemmas are all in the context of
Construction A3, with all assumptions there (recall especially the
ones in (1a) and (2a)) satisfied.
\bigskip

   {\bf Lemma A4.} \quad {\sl Suppose $H$ is a hinge in
$\tilde P_d$.  Then $H$ is 5-edge-connected in $G$ with the 5-cycle
$Q_0$ (in $\tilde G_0$).}
\bigskip

   {\it Proof.} \quad For each hinge $H$ in $\tilde P_d$, there
exist five edge-disjoint paths in the graph $\tilde P_d \cup \{w_1,
\dots, w_5\} \cup \{E_1, \dots, E_5\}$ that each start with a vertex
in $H$ and end with one of the edges $E_i$ (i.e.\ with one of the
vertices $w_i$).  (One can verify that fact for each of the 47
hinges in $\tilde P_d$, one by one.  By an obvious symmetry, it
suffices to check 24 of those hinges.)

   Recall that in $G_0$, the hinge $H_0$ and the 5-cycle
$Q_0$ are disjoint.  Since $G_0$ was cyclically 5-edge-connected (by
assumption), one has by Lemma A2 that there exist five edge-disjoint
paths in $G_0$ from $Q_0$ to $H_0$, i.e.\ ending with the edges
$e_1, \dots e_5$ respectively.

   By an obvious ``splicing'' of the edges $E_i$ and $e_i$ for
each $i = 1, \dots 5$, one obtains five edge-disjoint paths from
$Q_0$ to any given hinge $H$ in $\tilde P_d$.  Thus Lemma A4 holds.
\bigskip

   {\bf Lemma A5.} \quad {\sl If $A$ and $B$ are disjoint
cycles in $\tilde G_0$, then $A$ and $B$ are 5-edge-connected in
$G$.}
\bigskip

   {\it Proof.} \quad Recall the assumption that $G_0$ is
cyclically 5-edge-connected.   The cycles $A$ and $B$ are cyclically
5-edge-connected in $G_0$.  That is, there exists a family
$\varepsilon$ of five edge-disjoint paths between $A$ and $B$; some
of those paths may include one or both edges in the hinge $H_0$
and/or one or more of the edges $e_i$ attached to $H_0$.

   By adding to $\tilde G_0$ the edges $E_i,\ i=1, \dots, 5$
and certain particular (vertices and) edges in $\tilde P_d$, one
extends $\tilde G_0$ to a particular subgraph $F$ of $G$ such that
$F$ is conformal to $G_0$, and one obtains a family $\varepsilon$ of
five edge-disjoint paths in $F$ (hence in $G$) that run from $A$ to
$B$, respectively ``mirroring'' the paths in the family
$\varepsilon$. Thus Lemma A5 holds.
\bigskip

   {\bf Lemma A6.} \quad {\sl The graph $G$ is cyclically
5-edge-connected.}
\bigskip

   {\it Proof.} \quad Suppose instead that there exist disjoint
cycles $A$ and $B$ in $G$ such that $A$ and $B$ are not
5-edge-connected to each other in $G$.  We shall seek a
contradiction.

   Let $S$ be a cut set of four or fewer edges such
that the graph $G-S$ is disconnected, with $A$ and $B$ being in
different components.   Choose such a set $S$ with minimal
cardinality.

The argument will be divided into two cases according to whether or
not the 5-cycle $Q_0$ contains edges in $S$.
\medskip

   {\it Case 1. $Q_0$ contains no edges in $S$\/.}  (One or more
of the edges in $S$ may have endpoints in $Q_0$.)

   Then one of the cycles $A$ or $B$ --- without loss of
generality, say $A$ --- is in the component of $G-S$ different from
the one $Q_0$ is in. If $A$ itself is entirely in $\tilde G_0$, then
by Lemma A5 (applied to $A$ and $Q_0$), one obtains a contradiction.
Therefore, suppose instead (here in Case 1) that $A$ is not entirely
in $\tilde G_0$.

   Then by observation (3d) in Construction 3, $A$ contains a
hinge $H$ in $\tilde P_d$.  By Lemma A4, $H$ is 5-edge-connected in
$G$ to $Q_0$.  As a trivial corollary, the entire cycle $A$ is
5-edge-connected to $Q_0$ in $G$; and thereby again a contradiction
occurs.  Thus Case 1 is dismissed.
\medskip

{\it Case 2. $Q_0$ contains an edge of $S$\/.}

   Then (as an elementary if slightly hidden consequence of the
``minimal cardinality'' of $S$) $Q_0$ contains exactly two edges in
$S$, and they are not adjacent to each other. It follows that one of
the cycles $A$ or $B$ --- without loss of generality, say $A$ --- is
in the component $F$ of $G-S$ that contains exactly one edge $e$ of
$Q_0$.  If $H$ is any hinge in $F$ that is disjoint from $Q_0$, then
by a simple argument, $H$ cannot be 5-edge-connected to $Q_0$ in
$G$.

   Now suppose first that $A$ is entirely in $\tilde G_0$.
Then since $A$ has at least 5 edges and can include at most one edge
of $Q_0$ (the edge $e$ mentioned above), $A$ has two adjacent edges
that form a hinge $H$ disjoint from $Q_0$. Since $G_0$ is cyclically
5-edge-connected, $H$ and $Q_0$ are 5-edge-connected to each other
in $G_0$ by Lemma A2. By mimicking the (``conformal graph'') proof
of Lemma A5, one has that $H$ and $Q_0$ are 5-edge-connected to each
other in $G$.  But that contradicts the last sentence in the
preceding paragraph above.

   Therefore, suppose instead that the cycle $A$ is not entirely
in $\tilde G_0$. Then by observation (3d) in Construction A3, $A$
has a hinge $H$ in $\tilde P_d$.  By Lemma A4, $H$ is
5-edge-connected to $Q_0$ in $G$, again producing a contradiction.
Thus Case 2 is dismissed, and Lemma A6 is proved.
\bigskip

   {\bf Lemma A7.} \quad {\sl The graph $G$ is a cyclically
5-edge-connected snark.  Also, for any hinge $H$ in $G$, there
exists a 5-cycle $Q^*$ in $G$ such that $Q^*$ is disjoint from $H$.}
\bigskip

   {\it Proof.} \quad It was already observed in (3c) in
Construction A3 that $G$ has girth at least 5.  By Lemma A6, $G$ is
in fact cyclically 5-edge-connected.  The fact that $G$ cannot be
edge-3-coloured, is a special case of a much more general result of
Kochol \cite{Ko} at the heart of his ``superposition'' technique for
constructing ``big'' snarks from ``smaller'' ones.  The second
sentence in Lemma A7 was pointed out in (3f) in Construction A3.
Thus Lemma A7 holds.
\bigskip

   {\bf Remark A8.} \quad The following comments do not pertain
to construction A3 above.  Instead they involve Isaac's \cite{Is}
``dot product.''

   It is well known that if two disjoint snarks $G_1$ and $G_2$
are ``joined'' via a ``dot product,'' the resulting graph $G$ is a
snark.  In particular, $G$ will be ``cyclically 4-edge-connected''
and have no ``squares'' (in addition to failing to be
edge-3-colourable).

   At least in the special case of the ``symmetric'' version of
the dot product that is employed in [\cite{Br}, Section 2], the
proof is somewhat similar to, but considerably easier than, the
arguments above (and does not involve the use a any particular cycle
such as $Q_0$ above as a ``reference point'').  In place of a hinge
(in the arguments above), one uses just an edge.  As an analog of
Lemma A2 above, one shows that in a simple cubic graph which is
cyclically 4-edge-connected, any two disjoint edges are
``4-edge-connected to each other.''  The rest of the argument
(roughly analogous to Lemmas A4-A7 above) will be left here to the
reader.

\end{document}